\documentclass[12pt, reqno]{amsart}

\usepackage{amsmath,amssymb,amsthm,amsfonts,verbatim}
\usepackage{microtype}
\usepackage[all,2cell]{xy}
\usepackage{mathtools}
\usepackage{graphicx}
\usepackage{pinlabel}
\usepackage{hyperref}
\usepackage{mathrsfs}
\usepackage{color}
\usepackage[dvipsnames]{xcolor}
\usepackage{enumerate}
\usepackage{cite}
\usepackage{soul}
\usepackage{MnSymbol}

\CompileMatrices

\usepackage[top=1.2in,bottom=1.2in,left=1in,right=1in]{geometry}

\usepackage{hyperref} %Hyperlink reference
\hypersetup{          %Set up  for hyperlinks
	colorlinks=true, breaklinks, linkcolor=[RGB]{51 102 204}, filecolor=Orchid, urlcolor=[RGB]{51 102 204},
	citecolor=Orchid, linktoc=all, }
\usepackage[nameinlink]{cleveref}
\theoremstyle{plain}
\newtheorem{theorem}{Theorem}[section]
\newtheorem{maintheorem}{Theorem}

\newtheorem{proposition}[theorem]{Proposition}
\newtheorem{lemma}[theorem]{Lemma}

\newtheorem{corollary}[theorem]{Corollary}

\theoremstyle{definition}
\newtheorem{definition}[theorem]{Definition}

\newtheorem{remark}[theorem]{Remark}

\newcommand{\nc}{\newcommand}
\nc{\dmo}{\DeclareMathOperator}

\nc{\OO}{\mathcal{O}}
\nc{\cA}{\mathcal{A}}
\nc{\cB}{\mathcal{B}}
\nc{\sB}{\mathscr{B}}
\nc{\C}{\mathbb{C}}
\nc{\cC}{\mathcal{C}}
\nc{\sC}{\mathscr{C}}
\nc{\BB}{\mathbb{B}}
\nc{\LL}{\mathcal{L}}
\nc{\bd}{\mathbf{d}}
\nc{\DD}{\mathbb{D}}
\nc{\cD}{\mathcal{D}}
\nc{\sD}{\mathscr{D}}
\nc{\bF}{\mathbb{F}}
\nc{\cF}{\mathcal{F}}
\nc{\cG}{\mathcal{G}}
\nc{\cI}{\mathcal{I}}
\nc{\cK}{\mathcal{K}}
\nc{\cL}{\mathcal{L}}
\nc{\cM}{\mathcal{M}}
\nc{\bM}{\mathbf{M}}
\nc{\N}{\mathbb{N}}
\nc{\cN}{\mathcal{N}}
\nc{\cO}{\mathcal{O}}
\nc{\bp}{\mathbf{p}}
\nc{\bP}{\mathbb{P}}
\nc{\cP}{\mathcal{P}}
\nc{\Q}{\mathbb{Q}}
\nc{\R}{\mathbb{R}}
\nc{\cS}{\mathcal{S}}
\nc{\cT}{\mathcal T}
\nc{\cU}{\mathcal U}
\nc{\cX}{\mathcal{X}}
\nc{\cY}{\mathcal{Y}}
\nc{\Z}{\mathbb{Z}}
\nc{\disk}{\mathbb{D}}
\nc{\hyp}{\mathbb{H}}

\nc{\CP}{\mathbb{CP}}
\nc{\RP}{\mathbb{RP}}
\dmo{\Mod}{Mod}
\dmo{\PMod}{PMod}
\dmo{\LMod}{LMod}
\dmo{\Diff}{Diff}
\dmo{\Homeo}{Homeo}
\dmo{\dist}{dist}
\dmo\BDiff{BDiff}
\dmo\SO{SO}
\dmo\Hom{Hom}
\dmo\SL{SL}
\dmo\rank{rank}
\dmo\sig{sig}
\dmo\Out{Out}
\dmo\Aut{Aut}
\dmo\Inn{Inn}
\dmo\GL{GL}
\dmo\PGL{PGL}
\dmo\Gr{Gr}
\dmo\PSL{PSL}
\dmo\BHomeo{BHomeo}
\dmo\EHomeo{EHomeo}
\dmo\EDiff{EDiff}
\dmo\Disc{Disc}
\dmo\Aff{Aff}
\dmo\Spin{Spin}
\renewcommand{\Re}{\operatorname{Re}}

\renewcommand{\bar}{\overline}

\dmo\Teich{Teich}
\dmo\Fix{Fix}
\nc{\pair}[1]{\ensuremath{\left\langle #1 \right\rangle}}
\nc{\abs}[1]{\ensuremath{\left| #1 \right|}}
\nc{\action}{\circlearrowright}
\nc{\norm}[1]{\ensuremath{\left | \left | #1 \right | \right |}}
\nc{\abcd}[4]{\ensuremath{\left(\begin{array}{cc} #1 & #2 \\ #3 & #4 \end{array}\right)}}
\nc{\into}\hookrightarrow
\dmo{\Isom}{Isom}
\nc{\normal}{\vartriangleleft}
\dmo{\Vol}{Vol}
\dmo{\im}{Im}
\dmo{\Push}{Push}
\dmo{\Conf}{Conf}
\dmo{\UConf}{UConf}
\dmo{\PConf}{PConf}
\dmo{\Poly}{Poly}
\dmo{\PB}{PB}
\dmo{\id}{id}
\dmo{\Jac}{Jac}
\dmo{\Pic}{Pic}
\dmo{\Stab}{Stab}
\dmo{\Arf}{Arf}
\dmo{\End}{End}
\dmo{\Gal}{Gal}
\dmo{\lcm}{lcm}
\dmo{\ab}{ab}
\dmo{\opp}{op}
\dmo{\SU}{SU}
\dmo{\Sp}{Sp}
\dmo{\OT}{\Omega \mathcal{T}}
\dmo{\OM}{\Omega \mathcal{M}}
\dmo{\PH}{\mathbb{P}\mathcal{H}}
\dmo{\spin}{spin}
\dmo{\even}{even}
\dmo{\odd}{odd}
\dmo{\comp}{\mathcal{H}}
\dmo{\Mgk}{\mathcal{M}_{g, \underline{\kappa}}}
\dmo{\orb}{orb}
\dmo{\AJ}{AJ}
\dmo{\Ck}{\mathsf{C}(\underline{\kappa})}
\dmo{\Int}{Int}
\dmo{\pr}{pr}
\dmo{\lab}{lab}
\dmo{\Sym}{Sym}
\dmo{\Ann}{Ann}
\dmo{\Rad}{Rad}
\dmo{\Ind}{Ind}
\dmo{\Div}{Div}
\dmo{\Res}{Res}
\dmo{\Hur}{Hur}
\dmo{\vcd}{vcd}

\nc{\Span}[1]{\operatorname{Span}(#1)}

\renewcommand{\epsilon}{\varepsilon}

\renewcommand{\le}{\leqslant}

\nc{\margin}[1]{\marginpar{\scriptsize #1}}
\nc{\para}[1]{\medskip\noindent\textbf{#1.}}
\definecolor{myblue}{RGB}{102,153, 255}
\definecolor{myred}{RGB}{204,0,0}
\definecolor{mygreen}{RGB}{0,204,0}
\definecolor{myorange}{RGB}{255,102,0}
\definecolor{mypurple}{RGB}{138,43,226}
\definecolor{myyellow}{RGB}{255,204,0}
\nc{\red}[1]{\textcolor{myred}{#1}}
\nc{\blue}[1]{\textcolor{myblue}{#1}}

\nc{\discL}{\abs{L}\setminus \sD_L}

\author{Nick Salter}

\address{Department of Mathematics, University of Notre Dame, 255 Hurley Building, Notre Dame, IN 46556}
\email{nsalter@nd.edu}

\date{October 3, 2024}

\title[Monodromy kernels for discriminant complements]{Topological monodromy kernels for fundamental groups of discriminant complements}

\begin{document}
\begin{abstract}
    A linear system on a smooth complex algebraic surface gives rise to a family of smooth curves in the surface. Such a family has a {\em topological monodromy representation} valued in the mapping class group of a fiber. Extending arguments of Kuno, we show that if the image of this representation is of finite index, then the kernel is infinite. This applies in particular to linear systems on smooth toric surfaces and on smooth complete intersections. In the case of plane curves, we extend the techniques of Carlson--Toledo to show that the kernel is quite rich (e.g. it contains a nonabelian free group).
\end{abstract}

\maketitle

\section{Introduction}

Let $L$ be a very ample line bundle on a smooth complex projective surface $X$, and let $\abs{L}$ denote the associated complete linear system. The discriminant hypersurface $\sD_L \subset \abs{L}$ parametrizes the non-smooth curves in $\abs{L}$; consequently there is a family $\cX_L$ of smooth curves over $\discL$. Any such family carries a {\em topological monodromy representation}
\[
\rho_{top}: \pi_1(\discL) \to \Mod(C),
\]
where $C$ is a generic smooth curve in $\abs{L}$ and $\Mod(C)$ denotes its mapping class group, i.e. the group of isotopy classes of orientation-preserving diffeomorphisms. 

The groups $\pi_1(\discL)$ are in general poorly-understood, but are expected to have a rich theory, being higher-dimensional analogues of Artin's braid group. $\rho_{top}$ provides a window into their structure. The image of $\rho_{top}$ has been computed in several cases: toric surfaces \cite{CL1, CL2, saltertoric}, complete intersection surfaces \cite{nickishan}, as well as in the local setting of versal deformations of isolated plane curve singularities \cite{nickpablo}. 

In light of these results, it is natural to ask whether $\rho_{top}$ is injective. In the case of plane curve singularities, this was posed by Sullivan in the 1970's. Wajnryb \cite{wajnryb} showed that the topological monodromy of the $E_6$ singularity is not injective; this was subsequently extended in \cite{nickpablo} to obtain the same result for all isolated plane curve singularities of genus $g \ge 7$ not of type $A_n$ or $D_n$.

In the ``global'' setting of linear systems on smooth projective surfaces, much less is known. Kuno \cite{kuno} showed that in the case of plane quartic curves, $\rho_{top}$ has infinite kernel, and this was later computed precisely by R. Harris \cite{reid}. Our first main result shows that Kuno's methods can be applied in a broad range of settings.

\begin{maintheorem}\label{main:infinite}
    Let $X, L, \sD_L, C$ be as above. If the genus of $C$ is at least three and the image of the topological monodromy representation $\rho_{top}: \pi_1(\abs{L} \setminus \sD_L) \to \Mod(C)$ is of finite index in $\Mod(C)$, then $\ker(\rho_{top})$ is infinite.

    In particular, this holds for $L$ any ample line bundle on a smooth toric surface for which the generic section is non-hyperelliptic and of genus $g \ge 5$, and for $L = \cO(d)$ on any smooth complete intersection surface $X$, for any $d > 0$ (so long as $g(S) \ge 3$).
\end{maintheorem}

In the interest of not appearing to overstate our own contributions, we should remark that the bulk of \Cref{main:infinite} follows from Kuno's work. We think it is worthwhile to explicitly record the fact that finite-index monodromy implies infinite kernel, and that, following the monodromy computations of \cite{saltertoric,nickishan}, this can be established in a variety of cases.\\

Our second main result obtains a stronger version of the above theorem in the setting of plane curves. Say that a group $\Gamma$ is {\em large} if there is a homomorphism $f: \Gamma \to G$, where $G$ is a noncompact real semisimple algebraic group, and the image of $f$ is Zariski dense. By the Tits alternative, large groups always contain nonabelian free groups. In \cite{CT}, Carlson-Toledo study an analogue of our problem, generalizing to the setting of smooth hypersurfaces in any $\CP^{n+1}$ and considering only the homological monodromy action on $H_n(X)$, where $X$ is a smooth hypersurface in $\CP^{n+1}$ of degree $d$. They find that the kernel of the homological monodromy is large, except in the few sporadic cases where it is known to be finite.

Their approach is to study a second family associated to the family of smooth degree-$d$ hypersurfaces: the family of cyclic branched covers of $\CP^{n+1}$ {\em branched over} such hypersurfaces. They use the homological monodromy representation of this second family to detect elements in the kernel of the original representation, employing a mixture of Hodge theory and Lie theory to obtain largeness. Here, we show that the branched cover monodromy is in fact sensitive enough to detect elements in the {\em topological} monodromy kernel as well.

\begin{maintheorem}\label{main:large}
    In the case $X = \CP^2$ and $L = \cO(d)$ for $d \ge 4$, $\ker(\rho_{top})$ is {\em large} in the above sense.
\end{maintheorem}

Our proof proceeds by making the work of Carlson-Toledo sufficiently explicit to be able to analyze Wajnryb's kernel element from \cite{wajnryb}, using the monodromy result of \cite{saltertoric} to embed Wajnryb's local setting into our own.

\para{A remark on terminology} We will have need to employ the terminologies of both algebraic geometry and low-dimensional topology. Unfortunately, these are not compatible, owing to algebraic geometry reckoning in complex dimension and topology in real dimension - thus e.g. a ``surface'' to an algebraic geometer is a $4$-manifold to a topologist. We adopt the following conventions: ``algebraic'' will always precede algebro-geometric objects, so that a {\em smooth algebraic curve} is a topological $2$-manifold, an {\em algebraic surface} is a $4$-manifold, etc. We reserve the term ``surface'' for the topological notion. An embedded closed connected $1$-manifold in a surface will be called a {\em simple closed curve}. 

For the sake of concision, we will assume familiarity with the basic notions of singularity theory, in particular {\em nodes, vanishing cycles, cusps, distinguished bases}, as well as the $ADE$ classification. Our reference on these matters is the volume \cite{AGZV}; we will supply more precise pointers as necessary.

 \para{Organization} We prove \Cref{main:infinite} in \Cref{section:thmA}. The proof of \Cref{main:large} occupies the remaining three sections. In \Cref{section:planecurves}, we recall the notions of {\em $r$-spin structures} and their associated {\em $r$-spin mapping class groups} used in describing the image of $\rho_{top}$. In \Cref{section:cyclic}, we give a slight extension of a result of Carlson-Toledo, and the proof of \Cref{main:large} is then carried out in \Cref{section:thmB}.

\para{Acknowledgements} It is a pleasure to thank Yusuke Kuno and Anatoly Libgober for their interest in the project and their comments. The author is supported by NSF grant no. DMS-2338485.

\section{\Cref{main:infinite}: Monodromy kernel via vanishing signature}\label{section:thmA}

Below and throughout, $\cM_g$ denotes the moduli space of smooth algebraic curves of genus $g$, and $\Mod_g$ denotes the mapping class group of closed oriented surfaces of genus $g$. Topologically, $\Mod_g$ is defined as the group of isotopy classes of orientation-preserving diffeomorphisms, and arises algebro-geometrically as the orbifold fundamental group of $\cM_g$:
\[
\Mod_g = \pi_1^{orb}(\cM_g).
\]
If a specific surface $S$ is fixed, we will write $\Mod(S)$ to denote its mapping class group.

\Cref{main:infinite} is adapted from Kuno's arguments in \cite{kuno}, who obtained this result in the case of plane quartic curves (cf. Proposition 6.3 and the subsequent remark in {\em op. cit.}). There are two essential halves of the argument. 

\begin{enumerate}
    \item (cf. \cite[Theorem 4.1]{kuno}) The smooth sections of $L = \cO(d)$ give a classifying map $\rho: \discL \to \cM_g$ for which the induced map $\rho^*: H^2(\Mod_g;\Q) \to H^2(\discL; \Q)$ contains the nontrivial {\em Meyer signature class} $[\tau_g]$ in the kernel.

    \item (cf. \cite[Proposition 6.3]{kuno}) In the case $L = \cO(4)$, the homomorphism $\rho_{top}: \pi_1(\abs{\cO(4)} \setminus \sD_{\cO(4)}) \to \Mod_3$ is surjective. 
\end{enumerate}

We will not need to discuss the Meyer signature cocycle in detail. See \cite{kuno} (or the original \cite{meyer}) for more details, or \cite[Section 5.6]{FM} for a more thorough discussion of $H^2(\Mod_g;\Q)$. Kuno's result now follows from the lemma below.

\begin{lemma} \label{lemma:kernel}
    Let $f: X \to Y$ be a morphism of orbifolds satisfying the following properties:
    \begin{enumerate}
        \item $f_*: \pi_1^{orb}(X) \to \pi_1^{orb}(Y)$ is surjective,
        \item $f^*: H^2(Y;\Q) \to H^2(X;\Q)$ is not injective,
        \item $Y$ is an orbifold $K(\pi,1)$ space,
        \item $X$ and $Y$ are good orbifolds (there are finite covers with trivial orbifold structure).
    \end{enumerate}
    Then $K := \ker(f_*) \le \pi_1^{orb} X$ is infinite.
\end{lemma}
\begin{proof}
    By property (4), we may pass to manifold covers of $X$ and $Y$; properties (1)-(3) continue to hold in this setting, and infinitude of the kernel at this level implies the same of the original $f_*$. So without loss of generality, we assume below that $X$ and $Y$ are manifolds.

    Since $Y$ is a $K(\pi,1)$ space, $f$ admits a factorization
    \[
    \xymatrix{
        X \ar^<<<<{g}[r] \ar@/_1.5pc/_{f}[rr] & K(\pi_1(X),1) \ar^>>>>>{h}[r]  & Y,
    }
    \]
    inducing on cohomology the diagram
    \[
    \xymatrix{
    H^2(Y;\Q) \ar^<<<<{h^*}[r] \ar@/_1.5pc/_{f^*}[rr] & H^2(\pi_1(X);\Q) \ar^<<<<{g^*}[r] & H^2(X;\Q).
    }
    \]
    Since $K(\pi_1(X),1)$ can be constructed from $X$ by attaching cells of degrees $\ge 3$, it follows that $g_*: H_2(\pi_1(X);\Q) \to H_2(X;\Q)$ is surjective, and dually $g^*$ is injective. Thus if $f^*$ is not injective, necessarily $h^*$ is not injective. The five-term exact sequence for the group extension $1 \to K \to \pi_1(X) \to \pi_1(Y) \to 1$ identifies $\ker(h^*)$ with the image of the connecting map $\delta: H^1(K;\Q) \to H^2(Y;\Q)$. It follows that $H^1(K;\Q)$ must have positive rank, and so $K$ must be infinite.
\end{proof}

\begin{remark}
    There are some slight differences between our working environment and that of Kuno. First, we choose to work in the projective setting of $\abs{L} = \bP(H^0(X;L)^*)$, whereas Kuno's arguments take place in the affine setting. The connectivity results Kuno relies on (discussed in detail below) are equally valid in both settings. 
    
    Somewhat more substantially, Kuno actually shows something stronger than \Cref{main:infinite} for plane curves. To wit, the automorphism group $\PGL_3(\C)$ of $\CP^2$ acts on the parameter space of smooth plane curves $\abs{\cO(d)}\setminus \sD_{\cO(d)}$, and Kuno shows that the monodromy map from the orbifold fundamental group of the quotient $(\abs{\cO(4)}\setminus \sD_{\cO(4)}) / \PGL_3(\C)$ to $\Mod_3$ has infinite kernel. For the sake of uniformity, we will always work in the setting of {\em parameter} spaces, and do not account for the action of the automorphism group in the cases where $X$ has nontrivial automorphisms.
\end{remark}

In subsequent work \cite{kuno2}, Kuno extends the vanishing theorem \cite[Theorem 4.1]{kuno} invoked above to the much more general setting of families of curves embedded in a fixed projective variety $X$. When $X$ is an algebraic surface, his result specializes as follows.

\begin{proposition}[Kuno, cf. \cite{kuno2}, Theorem 1.0.2]\label{prop:meyertriv}
    Let $L$ be a very ample line bundle on a smooth projective algebraic surface $X$ for which a smooth section is an algebraic curve of genus $g$. Let $\abs{L}$ denote the complete linear system, $\sD_L$ the discriminant locus, and $\rho: \discL \to \cM_g$ the classifying map. Let $[\tau_g] \in H^2(\Mod_g;\Q)$ denote the signature class. Then $\rho^*([\tau_g]) = 0$.
\end{proposition}

\begin{proof}[Proof of \Cref{main:infinite}] 
Let $X$ be a smooth projective algebraic surface, and $L$ a very ample line bundle on $X$. By \Cref{prop:meyertriv}, $[\tau_g]$ is contained in the kernel of $\rho^*: H_2(\Mod_g;\Q) \to H_2(\discL; \Q)$. By hypothesis, a smooth section of $L$ has genus $g \ge 3$, so that $[\tau_g] \in H_2(\Mod_g;\Q)$ is nontrivial by \cite[Satz 2]{meyer}. By hypothesis, 
\[
\rho_*: \pi_1(\discL) \to \Mod_g
\]
has image of finite index in $\Mod_g$. Set $Y$ to be the finite-sheeted cover of $\cM_g$ classified by $\im(\rho_*)$. By transfer, the pullback of $[\tau_g]$ to $H_2(Y;\Q)$ remains nonzero, and so by \Cref{lemma:kernel}, $\rho_*$ has infinite kernel as claimed.

In the case of $X$ a smooth toric surface, the monodromy $\im(\rho_*) \le \Mod_g$ was shown to be of finite index in \cite[Theorem A]{saltertoric}, under the hypotheses that the general fiber is not hyperelliptic and of genus $g \ge 5$. In the case of complete intersection surfaces, the corresponding monodromy calculation was obtained in \cite[Theorem A]{nickishan}. 
\end{proof}

\section{Monodromy of plane curves}\label{section:planecurves}

We turn now to our second main result \Cref{main:large}. For this, we will require a brief recollection of the elements of the theory of {\em $r$-spin mapping class groups}, following the main references \cite{planequintics,saltertoric,strata2,strata3}.

\subsection{$r$-spin structures and winding number functions}
\begin{definition}
    Let $C$ be a smooth algebraic curve. An {\em $r$-spin structure} is a line bundle $L$ for which $L^{\otimes r} \cong K_C$.
\end{definition}

By the adjunction formula, a smooth plane curve $C$ of degree $d$ carries the distinguished $r$-spin structure $L = \cO_C(1)$; here $r = d-3$. It turns out that $r$-spin structures admit a purely topological classification.

\begin{definition}\label{definition:WNF}
    Let $S$ be a closed oriented surface of genus $g$, and let $\cS^+(S)$ denote the set of isotopy classes of oriented simple closed curves on $S$. A {\em $\Z/r\Z$-winding number function} is a set function
    \[
    \phi: \cS^+(S) \to \Z/r\Z
    \]
    that satisfies the following axioms:
    \begin{enumerate}
        \item (Reversibility) $\phi(\bar c) = - \phi(c)$, where $\bar c$ denotes $c$ endowed with the opposite orientation,
        \item (Twist-linearity) $\phi(T_c(b)) = \phi(b) + \pair{b,c}\phi(c)$, where $T_c$ denotes the Dehn twist about $c$ and $\pair{b,c}$ denotes the algebraic intersection number,
        \item (Homological coherence) If $S' \subset S$ is a proper subsurface with boundary components $c_1, \dots, c_k$, then $\sum \phi(c_i) = \chi(S')$, where $\chi(S')$ denotes the Euler characteristic, and all $c_i$ are oriented with $S'$ lying to the left. 
    \end{enumerate}
\end{definition}

\begin{lemma}
    Let $C$ be a smooth algebraic curve. There is a one-to-one correspondence between the set of $r$-spin structures on $C$ and the set of $\Z/r\Z$-winding number functions on $C$. 
\end{lemma}
\begin{proof}
    See \cite[Sections 2, 3]{planequintics}.
\end{proof}

\subsection{$r$-spin mapping class groups and monodromy of plane curves}
The mapping class group $\Mod(S)$ acts on the set of $\Z/r\Z$-winding number functions (and hence on the set of $r$-spin structures) by means of its action on $\cS^+(S)$. 
\begin{definition}
    Let $\phi$ be a $\Z/r\Z$-winding number function on a surface $S$. The associated {\em $r$-spin mapping class group} $\Mod(S)[\phi]$ is the stabilizer of $\phi$ under the action of $\Mod(S)$ on the set of $\Z/r\Z$-winding number functions. 
\end{definition}

In light of the correspondence between $r$-spin structures and $\Z/r\Z$-winding number functions, and in light of the fact that $\cO(1)$ determines a $d-3$-spin structure on all smooth plane curves of degree $d$, we obtain the containment
\[
\im(\rho_{top}) \le \Mod(C)[\phi],
\]
where $C$ is a smooth plane curve of degree $d$ and $\phi$ is the $\Z/(d-3)\Z$-spin structure associated to $\cO_C(1)$. 

\begin{theorem}\label{theorem:monodromy}
    For all $d \ge 4$, this containment is an equality:
    \[
    \im(\rho_{top}) = \Mod(C)[\phi].
    \]
\end{theorem}

\begin{proof}
    See \cite[Corollary 1.3]{strata2} or \cite[Theorem A]{nickishan}.
\end{proof}

\subsection{Change of coordinates} The monodromy theorem \Cref{theorem:monodromy} is a very useful tool for constructing vanishing cycles with prescribed topological properties. By the Picard-Lefschetz formula, if $a \subset C$ is a vanishing cycle and is nonseparating, then the Dehn twist $T_a$ about $a$ satisfies $T_a \in \im(\rho_{top})$. Such $T_a$ must necessarily preserve the $\Z/r\Z$-winding number function $\phi$, and from the twist-linearity condition of \Cref{definition:WNF}, it follows that $\phi(a) = 0$.  

Accordingly, we say that a nonseparating simple closed curve $a \subset C$ is {\em admissible} if $\phi(a) = 0$ (necessarily for either choice of orientation of $a$). A corollary of \Cref{theorem:monodromy} (cf. \cite[Lemma 11.5]{saltertoric}) is that {\em every} admissible curve arises from a vanishing cycle. 

We obtain \Cref{lemma:E6} below via an application of the ``change-of-coordinates principle for $r$-spin mapping class groups'' (cf. \cite[Proposition 2.15]{strata3}). A set of curves $a_0, \dots, a_5 \subset C$ is said to be {\em of type $E_6$} if the geometric intersection $i(a_i,a_j)$ is $1$ or $0$ according to whether the corresponding vertices $i,j$ in the Dynkin diagram of type $E_6$ are adjacent or not (thus, $i$ and $i+1$ are adjacent for $i = 1, \dots, 4$ and additionally $0$ and $3$ are adjacent).

\begin{lemma}\label{lemma:E6}
    Let $C$ be a smooth plane curve of degree $d \ge 4$. Then there is a configuration $a_0, a_1, \dots, a_5$ of vanishing cycles on $C$ of type $E_6$. 
\end{lemma}

\section{Cyclic branched covers}\label{section:cyclic}

Here we recall some of the elements of the work of Carlson-Toledo \cite{CT}. They consider the space 
\[
\cP_d := \abs{\cO(d)} \setminus \sD_{\cO(d)}
\]
of smooth plane curves. For every divisor $k$ of $d$, a curve $C \in \cP_d$ determines a $k$-fold cyclic branched covering $Y \to \CP^2$ with branch locus $C$. There are several associated monodromy representations. Set
\[
\rho_{top}: \pi_1(\cP_d) \to \Mod(C)
\]
to be the topological monodromy of the family of smooth plane curves of degree $d$, and set 
\[
\rho_{alg}: \pi_1(\cP_d) \to \Sp_{2g}(\Z)
\]
to be the induced action on $H_1(C;\Z)$. 
With $k \mid d$ understood, define
\[
\rho': \pi_1(\cP_d) \to \Aut(H_2(Y;\C))
\]
as the monodromy action on the family of $k$-fold cyclic branched covers. The cup product pairing $(\cdot, \cdot)$ on $H_2(Y; \Z)$ extends to a Hermitian form $h(x, y) = (x, \bar y)$ on $H_2(Y;\C)$. Since $Y$ carries an action of $\Z/k\Z$ that commutes with the action of $\rho'$, the eigenspace decomposition 
\[
H_2(Y;\C) = \bigoplus_{\lambda: \lambda^k = 1} H_2(Y)_\lambda
\]
is $\rho'$-invariant. Consequently, we define
\[
\rho'_\lambda:\ \pi_1(\cP_d) \to \Aut(H_2(Y)_\lambda)
\]
as the restriction of $\rho'$ to a chosen eigenspace.

The main results of this section are \Cref{prop:reflection} and \Cref{prop:alphabeta}. In the case of $k$ and $d$ both odd, these were obtained by Carlson-Toledo \cite[Proposition 6.1]{CT}. They are undoubtedly known to experts, but are included since we were unable to find a suitable reference.

\begin{proposition}\label{prop:reflection}
    Let $C \subset \CP^2$ be a smooth plane curve of degree $d$, and let $Y$ be the $k$-fold cyclic branched cover of $\CP^2$ branched along $C$, for some $k$ dividing $d$. Under a nodal degeneration of $C$ with vanishing cycle $\alpha \subset C$, the induced monodromy on $H_2(Y)_\lambda$ is given by a complex reflection
    \[
    x \mapsto x - (\lambda - 1)h(x,\eta_\alpha)\eta_\alpha
    \]
    for some element $\eta_\alpha \in H_2(Y)_\lambda$ determined by the vanishing cycle $\alpha \subset C$. 
\end{proposition}

\begin{remark}
    The construction of $\eta_\alpha$ from $\alpha$ depends on an arbitrary choice (the distinguished basis of vanishing cycles for the singularity in $Y$). It is not hard to see that different choices scale $\eta_\alpha$ by some $k^{th}$ root of unity; in particular, the formula of \Cref{prop:reflection} is well-defined.
\end{remark}

For simple closed curves $\alpha, \beta$ on a smooth algebraic curve $C$, the geometric intersection number of $\alpha$ and $\beta$ is notated, as is customary, by $i(\alpha,\beta)$.

\begin{proposition}\label{prop:alphabeta}
    With $C, Y$ as in \Cref{prop:reflection}, let $\alpha, \beta$ be vanishing cycles on $C$. If $i(\alpha,\beta) = 0$, then $h(\eta_\alpha, \eta_\beta) = 0$, and if $i(\alpha,\beta) = 1$, then $h(\eta_\alpha, \eta_\beta) = (1- \lambda)^{-1}$ (possibly after replacing $\eta_\beta$ with some equivalent $\pm \zeta \eta_\beta$ for $\zeta$ a $k^{th}$ root of unity).
\end{proposition}

\begin{corollary}\label{corollary:monformula}
    In the setting of \Cref{prop:alphabeta}, let $T_\alpha, T_\beta \in \Aut(H_2(Y)_\lambda)$ denote the monodromies associated to degeneration along $\alpha, \beta$ respectively. Then 
    \[
    T_\beta(\eta_\beta) = \lambda \eta_\beta.
    \]
    
    If $i(\alpha, \beta) = 1$ and $h(\eta_\alpha, \eta_\beta) = (1-\lambda)^{-1}$, then
    \[
    T_\beta(\eta_\alpha) = \eta_\alpha + \eta_\beta \quad \mbox{and} \quad T_\alpha(\eta_\beta) = \eta_\beta - \lambda \eta_\alpha.
    \]
\end{corollary}
\begin{proof}
    These can be directly computed from \Cref{prop:reflection}, using the determination of $h(\eta_\alpha,\eta_\beta)$ of \Cref{prop:alphabeta}.
\end{proof}

We will require some preliminary results.  We recall that the Picard-Lefschetz theorem for nodal singularities in complex dimension $2$ asserts that the monodromy of a nodal degeneration of an algebraic surface $Y$ is given by
\[
T_\delta(x) = x + (x,\delta)\delta,
\]
where $\delta \in H_2(Y;\Z)$ is the homology class of the vanishing cycle, satisfying $(\delta, \delta) = -2$, where as usual $(\cdot, \cdot)$ denotes the intersection pairing.

Let $\Delta \subset \C$ denote the closed unit disk. Recall that an isolated surface singularity is said to be {\em of type $A_{k-1}$} if it has local equation $x^2 + y^2 + z^k$.

\begin{lemma}\label{lemma:aksing}
    Let $\pi: \cY \to \Delta$ be a family of algebraic surfaces for which the fiber over $0$ has an isolated singularity of type $A_{k-1}$ and every other fiber is smooth. Set $Y = \pi^{-1}(1)$. Then there are vanishing cycles $\delta_1, \dots, \delta_{k-1} \in H_2(Y;\Z)$ such that 
    \[
    (\delta_i, \delta_j) = \begin{cases}  
                    -2 & i= j\\
                    1 & \abs{i-j} = 1\\
                    0 & \mbox{else},
    \end{cases}
    \]
    and such that the local monodromy on $H_2(Y;\Z)$ is given by
    \[
    T:= T_{\delta_{k-1}} \dots T_{\delta_1}.
    \]
    A formula for $T$ is given as
    \begin{equation}\label{equation:T}
    T(v) = v + (v,\delta_1) \delta_1 + (v, \delta_1 + \delta_2) \delta_2 + \dots + (v, \delta_1 + \dots + \delta_{k-1})\delta_{k-1}.
    \end{equation}
\end{lemma}
\begin{proof}
    This is a basic result in singularity theory; see \cite[Lemma 2.4, Section 2.9]{AGZV}. The explicit formula for $T$ follows by direct calculation.
\end{proof}

\begin{lemma}\label{lemma:sigmaVCaction}
    Suppose $C \subset \CP^2$ is a plane curve of degree $d$ with a nodal singularity. Let $Y$ be the $k$-fold cyclic branched cover of $\CP^2$ branched over $C$, for some $k \mid d$. Then $Y$ has a singularity of type $A_{k-1}$ over the node of $C$. With respect to the vanishing cycles $\delta_1, \dots, \delta_{k-1}$ of \Cref{lemma:aksing}, define $\delta_k := -\delta_1 - \dots - \delta_{k-1}$. Then the deck transformation $\sigma: Y \to Y$ satisfies
    \[
    \sigma(\delta_i) = \delta_{i+1},
    \]
    with subscripts taken mod $k$. 
\end{lemma}
\begin{proof}
   Near the nodal point of $C$, there are local coordinates $x,y$ on $\CP^2$ so that $C$ is given by the equation $x^2 + y^2 = 0$. Consequently the local equation for $Y$ is $x^2 + y^2 + z^k = 0$. This is the double suspension of the singularity $z^k = 0$ of type $A_{k-1}$, and hence is itself of type $A_{k-1}$. 

   To understand the interaction between $\sigma$ and the vanishing cycles, we recall (cf. \cite[Theorem 2.15]{AGZV}) that in the level manifold $\bar V:= \{z: z^k=1\}$ of the singularity $z^k = 0$, a distinguished set of vanishing cycles in $H_0(\bar V;\Z)$ is given by taking
   \[
   \bar{\delta}_i = \zeta^i-\zeta^{i-1}
   \]
   for $\zeta = e^{2 \pi i/k} \in \bar V$; the cycles $\bar \delta_1, \dots, \bar \delta_{k-1}$ give a basis. The relation
   \[
   \bar{\delta}_1 + \dots + \bar{\delta}_k = 0
   \]
   holds, and since the action of $\sigma$ on $\bar V$ is given by multiplication by $\zeta$, the required equality $\sigma(\bar{\delta}_i) = \bar{\delta}_{i+1}$ holds. Via the Sebastiani-Thom theorem (cf. \cite[Theorem 2.9]{AGZV}), for $\epsilon$ suitably small, the level set 
   \[
   V:= \{(x,y,z): x^2 + y^2 + z^k=\epsilon\}
   \]
   has the homotopy type of the double suspension of $\bar V$, and moreover from the proof it is clear this is equivariant for the actions of $\sigma$ on $\bar V$ and $V$. Defining $\delta_i$ as the double suspension of $\bar \delta_i$, it follows that $\sigma(\delta_i) = \delta_{i+1}$ holds in $H_2(V;\Z)$. By a careful choice of local level set $V$, the inclusion $V\into Y$ can be made $\sigma$-equivariant, showing that the same relations hold in $H_2(Y;\Z)$ as claimed.
   \end{proof}

\begin{lemma}\label{lemma:stcommute}
    The automorphisms $\sigma, T$ of $H_2(Y;\Z)$ commute.
\end{lemma}

\begin{proof}
For a diffeomorphism $f: Y \to Y$ and a Dehn twist $T_\delta$, there is an equality
\[
f_* T_\delta f_*^{-1} = T_{f_*(\delta)}
\]
of automorphisms of $H_2(Y;\Z)$. Thus by \Cref{lemma:sigmaVCaction},
\[
\sigma T \sigma^{-1} = \sigma (T_{\delta_{k-1}} \dots T_{\delta_1}) \sigma^{-1} = T_{\delta_k} \dots T_{\delta_2} = T_{\delta_k} T T_{\delta_1}^{-1}.
\]
From \eqref{equation:T}, $T(\delta_1) = \delta_k$,
so that
\[
\sigma T \sigma^{-1} =  T_{\delta_k} T T_{\delta_1}^{-1} = T_{\delta_k} T T_{\delta_1}^{-1} T^{-1} T = T_{\delta_k} T_{\delta_k}^{-1} T = T
\]
as claimed.
\end{proof}

Let $\lambda$ be a $k^{th}$ root of unity, and define the projection operator $P_\lambda: H_2(Y;\C) \to H_2(Y)_\lambda$ via
\[
P_\lambda(v) = \tfrac{1}{k}\sum_{i = 1}^k \bar \lambda^i \sigma^i.
\]

\begin{lemma}\label{lemma:hformula}
    Let $v \in H_2(Y)_\lambda$ and $w \in H_2(Y;\C)$ be given. Then
    \[
    h(v,P_\lambda(w)) = h(v,w).
    \]
\end{lemma}
\begin{proof}
    The action of $\sigma$ is Hermitian for $h$, and so
    \[
    h(v,\sigma w) = h(\sigma^{-1}v,w) = \bar \lambda h(v,w),
    \]
    since $v \in H_2(Y)_\lambda$. Since $h(\cdot, \cdot)$ is conjugate-linear in the second argument, the claim now follows from the definition 
    \[
    P_\lambda = \tfrac{1}{k}\sum_{i = 1}^k \bar \lambda^i \sigma^i.
    \]
\end{proof}

We are now in position to prove \Cref{prop:reflection}.

\begin{proof}[Proof of \Cref{prop:reflection}]
Let $\delta_1, \dots,\delta_{k-1}$ be the vanishing cycles in $Y$ associated to the nodal degeneration of $C$ with vanishing cycle $\alpha$, as in \Cref{lemma:aksing}. Define
\[
\xi_\alpha := P_\lambda(\delta_1).
\]
Since $\sigma, T$ commute (\Cref{lemma:stcommute}), $T$ restricts to an automorphism of each eigenspace $H_2(Y)_\lambda$. We will establish the formula
\[
T(v) = v + (\bar \lambda - 1) \frac{h(v,\xi_\alpha)}{h(\xi_\alpha, \xi_\alpha)} \xi_\alpha
\]
for the action of $T$ on $H_2(Y)_\lambda$. The formula of \Cref{prop:reflection} will then follow by normalizing $\xi_\alpha$ to a unit vector $\eta_\alpha$, with the sign change introduced by the fact (to be seen in \eqref{equation:xisign} below) that $h(\xi_\alpha, \xi_\alpha) < 0$. The factor $\bar \lambda$ can be replaced with $\lambda$ at the cost of exchanging $T$ for $T^{-1}$; the ensuing arguments are unaffected. 

Since $\delta_i = \sigma^{i-1}\delta_1$, the term $(v,\delta_1 + \dots + \delta_i)$ appearing in \eqref{equation:T} can be written as
\begin{align*}
    (v,\delta_1 + \dots + \delta_i) &= h(v,\delta_1 + \dots + \delta_i)\\
    &= h(v, (1 + \sigma + \dots + \sigma^{i-1})\delta_1)\\
    &= (1 + \bar \lambda + \dots \bar \lambda^{i-1})h(v, \delta_1)\\
    &= (1 + \bar \lambda + \dots \bar \lambda^{i-1})h(v, \xi_\alpha),
\end{align*}
with the latter two equalities holding by \Cref{lemma:hformula} and its proof. \eqref{equation:T} can now be rewritten as
\begin{align*}
    T(v) &= v + h(v, \xi_\alpha)(\delta_1 + (1+ \bar \lambda) \delta_2 + \dots + (1 + \dots + \bar\lambda^{k-2})\delta_{k-1})\\
    &= v + h(v, \xi_\alpha)(1 + (1 + \bar \lambda) \sigma + \dots + (1 + \dots + \bar \lambda^{k-2}) \sigma^{k-2})\delta_1.
\end{align*}

We claim that the operator
\[
Q_\lambda := (1 + (1 + \bar \lambda) \sigma + \dots + (1 + \dots + \bar \lambda^{k-2}) \sigma^{k-2})
\]
appearing above satisfies the identity
\[
Q_\lambda = \frac{k}{1-\lambda}P_\lambda.
\]
Indeed, since $\sum_{i = 0}^{k-1} \sigma^i = 0$ on $H_2(Y)_\lambda$,
\begin{align*}
    P_\lambda &= \tfrac{1}{k}(1 + \bar \lambda \sigma + \dots + \bar \lambda^{k-1} \sigma^{k-1})\\
    &= \tfrac{1}{k}( (1-\lambda) + (\bar \lambda - \lambda) \sigma + \dots + (\bar \lambda^{k-2} - \lambda)\sigma^{k-2}).
\end{align*}
The claim now follows from the identity 
\[
(\bar \lambda^i - \lambda) = (1-\lambda)(1 + \dots + \bar \lambda^{i}).
\]

Returning to $T(v)$, 
\begin{align*}
    T(v) &= v + h(v, \xi_\alpha)(1 + (1 + \bar \lambda) \sigma + \dots + (1 + \dots + \bar \lambda^{k-2}) \sigma^{k-2})\delta_1\\
    &= v + \frac{k}{1-\lambda}h(v, \xi_\alpha)P_\lambda(\delta_1)\\
    &= v + \frac{k}{1-\lambda}h(v, \xi_\alpha) \xi_\alpha\\
    &= v + \frac{k h(\xi_\alpha, \xi_\alpha)}{1 - \lambda} \frac{h(v, \xi_\alpha)}{h(\xi_\alpha, \xi_\alpha)}\xi_\alpha.
\end{align*}
To complete the proof, we will show that 
\[
\frac{k h(\xi_\alpha, \xi_\alpha)}{1 - \lambda} = (\bar \lambda-1),
\]
or equivalently,
\begin{equation} \label{equation:xisign}
    h(\xi_\alpha, \xi_\alpha) = \tfrac{1}{k}(1-\lambda)(\bar \lambda-1).
\end{equation}
Note that $\tfrac{1}{k}(1-\lambda)(\bar \lambda-1) = \tfrac{2}{k}(\Re(\lambda)-1) < 0$ as claimed above.

From \Cref{lemma:hformula}, $h(\xi_\alpha, \xi_\alpha) = h(\xi_\alpha, P_\lambda(\delta_1)) = h(\xi_\alpha, \delta_1)$. Expanding,
\begin{align*}
    \xi_\alpha = P_\lambda(\delta_1) &= \tfrac{1}{k}(\delta_1 + \bar \lambda \delta_2 + \dots + \bar \lambda^{k-1} \delta_k)\\
    &= \tfrac{1}{k}((1-\lambda) \delta_1 + (\bar \lambda - \lambda) \delta_2 + \dots + (\bar \lambda^{k-2}-\lambda) \delta_{k-1}).
\end{align*}
Then by \Cref{lemma:aksing}, 
\begin{align*}
    h(\xi_\alpha, \xi_\alpha) = h(\xi_\alpha, \delta_1) &= h(\tfrac{1}{k}((1-\lambda) \delta_1 + (\bar \lambda - \lambda) \delta_2 + \dots + (\bar \lambda^{k-2}-\lambda) \delta_{k-1}), \delta_1)\\
    &= \tfrac{1}{k}(-2(1-\lambda) + (\bar \lambda - \lambda))\\
    &= \frac{\lambda + \bar \lambda -2}{k}\\
    &= \frac{(1-\lambda)(\bar \lambda - 1)}{k}.
\end{align*}
\end{proof}

\begin{proof}[Proof of \Cref{prop:alphabeta}]
   Suppose first that $i(\alpha,\beta) = 0$. The distinguished basis of vanishing cycles $\delta_{\alpha,1}, \dots, \delta_{\alpha,k-1} \subset Y$ induced from the nodal degeneration of $C$ with vanishing cycle $\alpha$ are supported near the preimage of $\alpha$ in $Y$, and likewise for the vanishing cycles $\delta_{\beta,1}, \dots, \delta_{\beta,k-1}$ induced by degenerating $C$ along $\beta$. Since $i(\alpha,\beta) = 0$, it follows that $(\delta_{\alpha,i},\delta_{\beta,j}) = 0$ for all $i, j$, and hence $h(\eta_\alpha, \eta_\beta) = 0$ as claimed.

Next suppose that $i(\alpha,\beta) = 1$. We claim that there is a cuspidal degeneration of $C$ for which $\alpha, \beta$ forms a distinguished basis of vanishing cycles. To see this, observe that any cuspidal singularity of $C$ admits a distinguished basis $\gamma, \delta$ of vanishing cycles with $i(\gamma,\delta) = 1$. Moreover, with respect to the distinguished $r$-spin structure $\phi$ on $C$, each of $\gamma, \delta$ is admissible. By \Cref{theorem:monodromy}, the monodromy group of the family of smooth plane curves of degree $d$ is the full associated $r$-spin mapping class group $\Mod(C)[\phi]$ (here $r = d-3$). By the change-of-coordinates principle for $r$-spin mapping class groups (\cite[Proposition 2.15]{strata3}), $\Mod(C)[\phi]$ acts transitively on pairs of admissible curves $\gamma,\delta$ with $i(\gamma,\delta) = 1$. Let $f$ be a path in $\discL$ terminating at a cusp, and let $g \in \pi_1(\discL)$ be a loop whose monodromy takes the distinguished basis for this cuspidal singularity to $\alpha,\beta$. Then the concatenation $f g^{-1}$ induces a cuspidal degeneration for which $\alpha, \beta$ is a distinguished basis of vanishing cycles as claimed. 

A local equation for a cusp is given by $x^2 + y^3 = 0$, and so a local equation for the corresponding singularity of $Y$ is given by $x^2 + y^3 + z^k = 0$. Once more following \cite[Section 2.9]{AGZV}, there is a distinguished basis of vanishing cycles of the following form: the cycles are given as
\[
\delta_{\alpha,1}, \dots, \delta_{\alpha,k-1}, \delta_{\beta,1}, \dots, \delta_{\beta,k-1},
\]
and the intersection form is described as
\begin{align*}
    (\delta_{\alpha, i}, \delta_{\alpha,i}) = (\delta_{\beta, i}, \delta_{\beta,i})= -2,\\
    (\delta_{\alpha,i},\delta_{\alpha,i+1}) = (\delta_{\beta,i},\beta{\alpha,i+1}) = 1,\\
    (\delta_{\alpha,i}, \delta_{\beta,i}) = 1,\\
    (\delta_{\alpha,i}, \delta_{\beta,i+1}) = -1,
\end{align*}
with all other intersections zero. Moreover, $\delta_{\alpha,1}, \dots, \delta_{\alpha,k-1}$ form a distinguished basis of vanishing cycles for the nodal degeneration of $C$ with vanishing cycle $\alpha$, and likewise for $\beta$. With $\xi_\alpha, \xi_\beta$ as in \Cref{prop:reflection}, a direct computation with the intersection form then shows that
\[
h(\xi_\alpha, \xi_\beta) = \frac{1-\lambda}{k}.
\]
Normalizing by $h(\xi_\alpha, \xi_\alpha) = h(\xi_\beta, \xi_\beta) = \frac{(1-\lambda)(\bar \lambda - 1)}{k}$, the formula
\[
h(\eta_\alpha, \eta_\beta) = \frac{1}{1-\bar \lambda}
\]
now follows.
\end{proof}

\section{\Cref{main:large}: large monodromy kernel for plane curves}\label{section:thmB}

In this section we complete the proof of \Cref{main:large}. We first establish some notation. Recall from \Cref{section:cyclic} the monodromy representations
\[
\rho_{top}: \cP_d \to \Mod(C)
\]
and
\[
\rho_{alg}: \cP_d \to \Sp_{2g}(\Z).
\]
Define 
\[
K_d^{top} = \ker(\rho_{top}) \qquad \mbox{and} \qquad K_d^{alg} = \ker(\rho_{alg}).
\]

\subsection{Constructing monodromy kernel}

The objective of this subsection is to prove the proposition below.

\begin{proposition}\label{prop:infiniteimg}
    For all $d \ge 4$ and all $\lambda$ a $k^{th}$ root of unity for $k \mid d$ and $k \ge 4$, there is $W \in K_d^{top}$ such that $\rho'_\lambda(W) \in \Aut(H_2(Y)_\lambda)$ has infinite order.
\end{proposition}

\para{Constructing the element} The particular element $W$ we study was discovered by Wajnryb \cite{wajnryb}, who used it to show that the topological monodromy of an isolated plane curve singularity of type $E_6$ is not injective. Here we see that Wajnryb's element serves the same purpose in our setting.

Let $a_0, \dots, a_5$ be the set of vanishing cycles of \Cref{lemma:E6} arranged in the $E_6$ configuration. Let $\eta_i \in H_2(Y)_\lambda$ be the class corresponding to $a_i$ as in \Cref{prop:reflection}, and let $M_i \in \Aut(H_2(Y)_\lambda)$ be the corresponding monodromy element. Normalize $\eta_0, \dots, \eta_5$ so that whenever $i < j$ and $h(\eta_i, \eta_j) \ne 0$,
\[
h(\eta_i,\eta_j) = (1-\lambda)^{-1}.
\]

    Set
    \[
    B := M_3 M_2 M_4 M_3 M_1 M_5 M_2 M_4 M_3,
    \]
    \[
    M_0' := B M_0 B^{-1},
    \]
    and
    \[
    W := (M_0 M_0' M_0)(M_0' M_0 M_0')^{-1}.
    \]

\begin{lemma}\label{lemma:wkernel}
    For all $d \ge 4$, $\rho_{top}(W) = 1$ in $\Mod_g$. 
\end{lemma}
\begin{proof}
    A direct calculation shows that $b:= \rho_{top}(B)$ takes the curve $a_0$ to a curve $a_0'$ for which $i(a_0, a_0') = 1$. Consequently $\rho_{top}(M_0') = b T_{a_0} b^{-1} = T_{a_0'}$ satisfies the braid relation with $\rho_{top}(M_0) = T_{a_0}$; equivalently $\rho_{top}(W) = 1$.
\end{proof}

\para{Certifying nontriviality} 
Let $V \le H_2(Y)_\lambda$ denote the span of $\eta_0, \dots, \eta_5$, and note that $V$ is invariant under $M_0, \dots, M_5$. It therefore suffices to show that the restriction of $\rho'_\lambda(W)$ to $V$ has infinite order. As a standing convention, we write $\bar{M_i}$ to denote the restriction of $\rho'_\lambda(W)$ to $V$, and extend this convention to include the products $B, M_0'$, and $W$.  
\begin{lemma}\label{lemma:eigenvalues}
    The eigenvalues of $\bar W$ are $1$ (with multiplicity $4$) and the roots of $x^2 + P(\lambda) x + 1$ for 
    \begin{align*}
    P(\lambda) = \frac{1}{\lambda^9}(&-\lambda ^{18}+3 \lambda ^{17}-\lambda ^{16}-8 \lambda ^{15}+13 \lambda ^{14}+\lambda ^{13}-23 \lambda ^{12}+20 \lambda ^{11}+12 \lambda ^{10}-34 \lambda ^9 \\ &+12 \lambda ^8+20 \lambda ^7-23 \lambda ^6+\lambda ^5+13 \lambda ^4-8 \lambda ^3-\lambda ^2+3 \lambda -1).
    \end{align*}
\end{lemma}
\begin{proof}
    This follows by a direct computation in Mathematica - see \cite{supplemental}.
\end{proof}

\begin{lemma}
    Let $\lambda = e^{2 \pi i p/q}$ be a root of unity for $1/4 \le p/q \le 2/5$. Then $\bar W$ has infinite order. 
\end{lemma}
\begin{proof}
    Following \Cref{lemma:eigenvalues}, it suffices to show that the roots of $x^2 + P(\lambda)x + 1$ are not roots of unity under the hypotheses on $\lambda$. Suppose to the contrary that some root $\mu$ of $x^2 + P(\lambda)x + 1$ is a root of unity. Then $\mu^{-1} = \bar \mu$ is the other root, and in particular, $P(\lambda) = -(\mu + \bar \mu) = -2 \Re(\mu)$ must be real, so that $\abs{P(\lambda)} \le 2$. But Mathematica shows that $P(e^{i \theta}) < -2$ on the interval $\theta \in [\pi/2, 4\pi/5]$ (again, see \cite{supplemental}).
\end{proof}

\subsection{Finishing the proof}
\begin{proof}[Proof of \Cref{main:large}]
    The work of Carlson-Toledo shows that $\rho_\lambda'(K_d^{alg}) \le \Aut(H_2(Y)_\lambda)$ is Zariski dense so long as the image is infinite and generated by complex reflections (see in particular the proof outline in \cite[Section 2]{CT}). They establish infinitude of the image for $\lambda = e^{2 \pi i / d}$ for all $d \ge 4$ in \cite[Proposition 5.1]{CT} via Hodge-theoretic methods; note that \Cref{prop:infiniteimg} gives a direct proof of this for $\lambda = e^{2 \pi i p/ q}$ with $1/4 \le p/q \le 2/5$. Generation by complex reflections was established for $k,d$ odd in \cite[Proposition 6.1]{CT} and in general in \Cref{prop:reflection} above. 

    Basic arithmetic shows that for all $d \ge 4$, there is some rational number $p/d$ (not necessarily in lowest terms) on the interval $[1/4, 2/5]$. Set $\lambda = e^{2 \pi i p/d}$ for such $p/d$. Arguing as in \cite[Section 2]{CT}, by passing to a finite-index subgroup $K' \le K_d^{hom}$, we can arrange for the image of $\rho_\lambda'(K')$ in the adjoint representation $\bar{\Aut(H_2(Y)_\lambda)}$ to lie in the identity component $\bar{\Aut(H_2(Y)_\lambda)}_0$ in the Zariski topology. For notational simplicity, following Carlson-Toledo, define $\bar G ':= \bar{\Aut(H_2(Y)_\lambda)}_0$, and define $\bar \rho': K' \to \bar G'$ as the induced representation.
    
    Carlson-Toledo then show that $\bar \rho'(K')$ is Zariski dense in the simple algebraic group $\bar G '$. By \Cref{prop:infiniteimg}, the image $\bar \rho '(K' \cap K_d^{top}) \le \bar G'$ is infinite and normal in $\bar \rho'(K')$. As $\bar G'$ is simple, it follows that $\bar \rho' (K' \cap K_d^{top})$ is likewise Zariski dense in $\bar G'$, as was to be shown.
\end{proof}

\bibliographystyle{alpha}
\bibliography{bibliography}

\end{document}